\documentclass[11pt]{amsart}

\usepackage{amsmath,amssymb,amsthm}
\usepackage{fullpage,url}
\usepackage[all]{xy}
\usepackage{color}
\usepackage{versions}

\renewcommand{\bar}[1]{#1\llap{$\overline{\phantom{\rm#1}}$}}

\DeclareMathOperator{\Spec}{Spec}

\DeclareMathOperator{\Hom}{Hom}

\DeclareMathOperator{\id}{id}

\DeclareMathOperator{\Per}{Per}
\DeclareMathOperator{\PGL}{PGL}
\DeclareMathOperator{\GL}{GL}

\DeclareMathOperator{\Gal}{Gal}
\DeclareMathOperator{\Rat}{Rat}
\DeclareMathOperator{\Twist}{Twist}

\DeclareMathOperator{\Br}{Br}
\DeclareMathOperator{\Conj}{Conj}

\newcommand{\col}{\,{:}\,}
\newcommand{\tth}{^{\operatorname{th}}}

\theoremstyle{plain}
\newtheorem{thm}{Theorem}

\newtheorem{prop}[thm]{Proposition}

\theoremstyle{definition}
\newtheorem{defn}[thm]{Definition}

\newtheorem{exmp}[thm]{Example}
\newtheorem*{rem}{Remark}
\theoremstyle{remark}

\def\Q{\mathbb{Q}}
\def\P{\mathbb{P}}

\def\calA{\mathcal{A}}

\excludeversion{code}

\begin{document}
\title{The field of definition for dynamical systems on $\mathbb{P}^{N}$}
\author[Hutz, Manes]
{Benjamin Hutz and Michelle Manes}

\address{Department of Mathematical Sciences \\
        Florida Institute of Technology \\
        Melbourne, FL 32901 \\
        USA}
\email{bhutz@fit.edu}

\address{Department of Mathematics \\
        University of Hawaii \\
        Honolulu, HI 96822 \\
        USA}
\email{mmanes@math.hawaii.edu}

\begin{abstract}
    Let $\Hom^N_d$ be the set of morphisms $\phi: \mathbb{P}^{N} \to \mathbb{P}^{N}$ of degree $d$.  For $f \in \PGL_{N+1}$, let $\phi^f = f^{-1} \circ \phi \circ f$ be the conjugation action and let $M^N_d = \Hom_d^N/\PGL_{N+1}$ be the moduli space of degree $d$ morphisms of $\P^N$. A \emph{field of definition} for $\chi \in M_d^N$ is a field over which at least one morphism $\phi \in \chi$ is defined. The \emph{field of moduli} for $\chi$ is the fixed field of $\{\sigma \in \Gal(\bar{K}/K) \col \chi^{\sigma} = \chi\}$. Every field of definition contains the field of moduli. In this article, we give a sufficient condition for the field of moduli to be a field of definition for morphisms whose stabilizer group  is trivial. 
\end{abstract}

\thanks{The  second author was partially supported by NSF grant  DMS-1102858. }


\maketitle

 \section{Introduction}
    We begin by fixing some notation:  $K$ is a field with fixed separable closure $\overline{K}$, and     $G_K = \Gal(\overline{K}/K)$.  Let $\Rat^N_d$ be the set of rational maps
    $\mathbb{P}^{N} \to \mathbb{P}^{N}$ of
    (algebraic) degree $d$ defined over $\overline K$ (meaning that $\phi$ is given in coordinates by homogeneous polynomials of degree~$d$), and   $\Hom_d^N\subset \Rat_d^N$ is the subset of morphisms.  We write $\Rat^N_d(K)$ and $\Hom_d^N(K)$ to indicate the subset of rational maps (respectively morphisms) that can be written in each coordinate as homogeneous polynomials with coefficients in~$K$.

    Let $\phi \in \Hom_d^N$.  For $f \in \PGL_{N+1}$, let $\phi^f = f^{-1} \circ \phi \circ f$ be the conjugation action and let $M^N_d = \Hom_d^N/\PGL_{N+1}$ be the quotient space for this action.  We know that $M_d^N$ exists as a geometric quotient \cite{levy}; hence, we call $M_d^N$ the moduli space of dynamical systems of degree $d$ morphisms of $\P^N$, and we will denote the conjugacy class of $\phi$ in $M_d^N$ by $[\phi]$.

    \begin{defn}
        Let $\phi \in \Hom_d^N$.
        A field $K'/K$ is a \emph{field of definition} for $\phi$ if
        $\phi^f \in \Hom_d^N(K')$ for some $f \in \PGL_{N+1}$.
    \end{defn}

    Of course, a given map $\phi$ has many fields of definition, but we are generally interested in finding a minimal such field.
    We can also define
    a minimal field for $\phi$ in a Galois theoretic sense.  First note that an element
    $\sigma \in G_K$ acts on a map $\phi \in \Hom_d^N$ by acting on its
    coefficients.  We denote this left action by $\phi^{\sigma}$.  Using Hilbert's
    Theorem 90, we see that $\phi$ is defined over $K$ if and only if $\phi$ is
    fixed by every element in $G_K$.  Furthermore, if $\psi = \phi^f$ then $\psi^{\sigma} = (\phi^\sigma)^{f^{\sigma}}$.  So the action of $G_K$ on
    $\Hom_d^N$ descends to an action on $M_d^N$.  We make the following
    definition.

    \begin{defn}
        Let $\phi \in \Hom_d^N$, and define
        \begin{equation*}
           G_\phi =  \{\sigma \in G_K \mid \phi^{\sigma} \text{ is } \overline K \text{ equivalent to } \phi\}.
        \end{equation*}
        The \emph{field of moduli} of $\phi$ is the fixed field $\overline K^{G_\phi}$.
    \end{defn}

    In other words, the field of moduli of $\phi$ is the smallest field $L$ with the
    property that for every $\sigma \in\Gal(\overline{K}/L)$ there is some $f_{\sigma} \in \PGL_{N+1}$ such that $\phi^{\sigma} = \phi^{f_{\sigma}}$. It is not hard to see that the field of moduli of $\phi$ must be contained in any field of definition of $\phi$.
    A fundamental question is:
    \begin{center}
    	\emph{When is the field of moduli also a field of definition?}
    \end{center}

    For $[\phi] \in M_d^1$, Silverman~\cite{Silverman12} proved that when the degree of $\phi$ is even, or when $\phi$ is a polynomial map of any degree,  the field of moduli is also a field of definition.
        In the present work, we solve the higher dimension problem for maps with a trivial stabilizer group.

            \begin{defn}
        For any $\phi \in \Hom_d^N$ define $\mathcal{A}_\phi$ to be the \emph{stabilizer group}  of $\phi$, i.e.,
        \[
            \calA_\phi = \{f \in \PGL_{N+1} \mid \phi^f = \phi\}.
        \]
    \end{defn}
           From~\cite{petsche}, $\mathcal{A}_{\phi}$ is well-defined as a finite subgroup of $\PGL_{N+1}$.  For a conjugacy class $[\phi] \in M_d^N$, however, the stabilizer group is only defined up to conjugacy class in $\PGL_{N+1}$.
  For most $\phi \in \Hom_d^N$ (all but a Zariski closed subset),  $\mathcal{A}_\phi$ is trivial.

\begin{thm} \label{thm.main}
        Let $[\phi] \in M_d^N(K)$, with $\mathcal{A}_{\phi} = \{1\}$. Let $D_n = \sum_{i=0}^N d^{ni}$. If there is a positive integer $n$ such that $\gcd(D_n,N+1)=1$, then $K$ is a field of definition for $[\phi]$.
	\end{thm}

        While the proof of the one-dimensional result in~\cite{Silverman12} was different in the cases of trivial versus nontrivial stabilizer group, the end result was independent of $\mathcal{A}_{\phi}$. In higher dimensions, this is not the case. The map
	\begin{equation*}
		\phi = [i(x-y)^3, (x+y)^3, z^3] \in \Hom_3^2		
	\end{equation*}
	has field of moduli strictly smaller than any field of definition.  Specifically, the field of moduli for $[\phi]$ is $\Q$, but $\Q$ is not a field of definition.  The map seems to satisfy the hypotheses of Theorem~\ref{thm.main} since $\gcd(D_1,N+1) =1$.  However, it is a simple matter to check that $\mathcal{A}_{\phi} = C_2$, the cyclic group of order two.  So the hypothesis on the stabilizer group in Theorem~\ref{thm.main} is, in fact, essential. (See Section~\ref{sect_examples} for more on this example.)

   Silverman \cite{Silverman12} also shows that his result is the best possible by giving an example in every odd degree of a map with field of moduli strictly smaller than any field of definition, as follows.

        \begin{exmp}\label{exmp1}
        Let $\alpha(z) = i\left(\frac{z-1}{z+1}\right)^d$ with $d$ odd.  Clearly $\Q(i)$ is a field of definition for $\alpha$.  Furthermore, if we let $\sigma$ represent complex conjugation, then we can check that
        \[
        \alpha^{\sigma} = \alpha^f \qquad \text{ for } \qquad f(z) = -\frac 1 z.
        \]
        Hence, $\Q$ is the field of moduli for $\alpha$.  But $\Q$ cannot be a field of definition for $\alpha$ because any such field must have $-1$ as a sum of two squares.  (See~\cite{Silverman12} for details.)          
    \end{exmp}

    For  $\phi \in \Hom_d^N$, the field of moduli $K'$ is also a field of definition precisely when $[\phi]$ is in the image of the
    projection $\Hom_d^N(K') \to M_d^N(K')$.  Example \ref{exmp1} shows that this map may not always be surjective.  We provide higher dimensional examples of this phenomenon in Section~\ref{sect_examples}.

\subsection*{Outline}
   In Section~\ref{sect.first.constructions}, we generalize to $\mathbb{P}^{N}$ some results from \cite{Silverman12} regarding morphisms on $\mathbb{P}^1$. The methods are essentially the same as in \cite{Silverman12}, with the details adjusted for arbitrary dimension. The proof of Theorem~\ref{thm.main} occupies Section~\ref{sect.trivial.stabilizer}.   Section~\ref{sect_examples} provides  details for the two examples described above.

\section{First Constructions} \label{sect.first.constructions}
 We begin this section by briefly developing the idea of $K$-twists of rational maps.

        \begin{defn}
        Two maps $\phi,\psi \in \Hom_d^N(K)$ are \emph{$K$-equivalent} if $\psi = \phi^f$ for some $f \in \PGL_{N+1}(K)$.

    \end{defn}
            Maps that are $K$-equivalent
         have the same arithmetic dynamical behavior over $K$.  In particular, $f$ maps $\psi$-orbits of $K$-rational points to $\phi$-orbits of $K$-rational points.  Further, the field extension of $K$ generated by the period-$n$ points of $\phi$ and $\psi$ must agree for every $n\geq 1$.
However, maps that are $\overline{K}$-equivalent but not $K$-equivalent may exhibit very different dynamics on $\P^N(K)$.

    \begin{defn}
        Let $\phi \in \Hom_d^N(K)$.  The map $\psi \in \Hom_d^N(K)$ is a \emph{$K$-twist} of $\phi$ if $\psi$ and $\phi$ are $\overline{K}$-equivalent.  Further, we define
        \begin{align*}
             \Twist_K(\phi)
             &= \frac{ \{K\text{-twists of } \phi\}}{K\text{-equivalence}} \\
             &=
             \{ \text{$K$-equivalence classes of maps
             defined over $K$ and $\overline K$-equivalent to $\phi$} \}.
        \end{align*}
    \end{defn}

    \begin{exmp}
        Let
        $$\phi(z) = 2z + \frac 5 z \quad  \text{ and } \quad \psi(z) = \frac{z^2-3z}{3z-1}.$$
          One can check that
        $$
        \phi^f(z) = \psi(z) \quad \text{ where } \quad f(z) = \frac{i\sqrt 5 \left(z - 1\right)}{1+z},
        $$
        so these maps are $\Q$-twists of each other.  However, they cannot be $\Q$-equivalent because the finite fixed points of $\psi(z)$ are rational (they are $z = -1$ and $z=0$),  but the finite fixed points of $\phi(z)$ are $\pm i\sqrt 5$.
    \end{exmp}

    \begin{rem}
        For $\phi,\psi \in \Hom_d^N(K)$,  $\psi$ is a $K$-twist of $\phi$ if and only if $[\phi]=[\psi]$ in $M_d^N$.  Thus, the study of twists may be viewed as the study of the fibers of $$[\cdot]: \Hom_d^N(K)/\PGL_{N+1}(K) \to M_d^N(K),$$ where we define
        \begin{equation}\label{eqn.mdk}
            M_d^N(K) = \left\{\xi \in M_d^N \colon \forall \sigma \in G_K,  \xi^{\sigma} = \xi \right\}.
        \end{equation}
    \end{rem}
One checks easily (or see~\cite[Proposition~$4.73$]{ads})  that if $\calA_\phi = \{\id\}$, then $\#\Twist(\phi) = 1$.

    \begin{prop} \label{prop1}
        Let $\xi \in M_d^N(K)$ be a dynamical system whose field of moduli is
        contained in $K$, and let $\phi \in \xi$ be any representative of
        $\xi$.
        \begin{enumerate}
            \item \label{prop1_item1} For every $\sigma \in G_K$, there exists an $f_{\sigma} \in
                \PGL_{N+1}$ such that
                \begin{equation*}
                    \phi^{\sigma} = \phi^{f_{\sigma}}.
                \end{equation*}
                The map $f_{\sigma}$ is determined by $\phi$ up
                to \textup(left\textup) multiplication by an element of $\mathcal{A}_{\phi}$.

            \item \label{prop1_item2} Having chosen $f_{\sigma}$'s as in \textup(\ref{prop1_item1}\textup), the resulting map
                $f:G_K \to \PGL_{N+1}$ satisfies
                \begin{equation*}
                    f_{\sigma}f_{\tau}^{\sigma}f^{-1}_{\sigma\tau} \in
                    \mathcal{A}_{\phi} \quad \text{for all } \sigma, \tau \in G_K.
                \end{equation*}
                We say that $f$ is a $G_K$-to-$\PGL_{N+1}$ cocycle relative to
                $\mathcal{A}_{\phi}$.

            \item \label{prop1_item3} Let $\Phi \in \xi$ be any other representative of
                $\xi$, and for each $\sigma \in G_K$ choose an automorphism $F_{\sigma}
                \in \PGL_{N+1}$ as in \textup(\ref{prop1_item1}\textup) so that $\Phi^{\sigma} = \Phi^{F_{\sigma}}$.
                Then there exists a $g \in \PGL_{N+1}$ such that
                \begin{equation} \label{eq1}
                    g^{-1}F_{\sigma}g^{\sigma}f^{-1}_{\sigma} \in \mathcal{A}_{\phi}\quad
                    \text{for all } \sigma \in G_K.
                \end{equation}
                We say that $f$ and $F$ are $G_K$-to-$\PGL_{N+1}$ cohomologous
                relative to $\mathcal{A}_{\phi}$.

            \item \label{prop1_item4} The field $K$ is a field of definition for $\xi$ if and only
                if there exists a $g \in \PGL_{N+1}$ such that
                \begin{equation*}
                    g^{-1}g^{\sigma}f^{-1}_{\sigma} \in \mathcal{A}_{\phi} \quad
                    \text{for all } \sigma \in G_K.
                \end{equation*}
            \end{enumerate}
    \end{prop}

    \begin{rem}
        For $\mathcal{A}_{\xi} = 1$, Proposition \ref{prop1} says that the map $f:
        G_K \to \PGL_{N+1}$ is a one-cocycle whose cohomology class in the cohomology
        set $H^{1}(G_K,\PGL_{N+1})$ depends only on $\xi$.  On the other hand, if
        $\mathcal{A}_{\xi}\neq1$, then the criterion in equation (\ref{eq1}) is not an equivalence
        relation, so we cannot even define a ``cohomology set relative to
        $\mathcal{A}_{\phi}$.''
    \end{rem}

    \begin{proof}
       \begin{enumerate}
       	\item
            Fix $\sigma \in G_K$.  By definition of $M_d^N(K)$ in (\ref{eqn.mdk}),
            there exists an
            $f_{\sigma}$ such that
            \begin{equation*}
                \phi^{\sigma} = \phi^{f_{\sigma}}.
            \end{equation*}
            Suppose that $g_{\sigma} \in \PGL_{N+1}$ has the same
            property.  Then
            \begin{equation*}
                \phi^{g_{\sigma}} = \phi^{\sigma} = \phi^{f_{\sigma}},
            \end{equation*}
            and consequently
            \begin{equation*}
                g_{\sigma}f_{\sigma}^{-1} \in \mathcal{A}_{\phi}.
            \end{equation*}

        \item
            Let $\sigma, \tau \in G_K$.  We compute
            \begin{equation*}
                \phi^{f_{\sigma\tau}} = \phi^{\sigma\tau} =
                (\phi^{\tau})^{\sigma} = (\phi^{f_{\tau}})^{\sigma} =
                (\phi^{\sigma})^{f^{\sigma}_{\tau}} =
                \phi^{f_{\sigma}f^{\sigma}_{\tau}}.
            \end{equation*}
            Hence, $f_{\sigma}f_{\tau}^{\sigma}f_{\sigma\tau}^{-1} \in
           \mathcal{A}_{\phi}$.

        \item
            Since $[\phi]=\xi=[\Phi]$, we can find some $g \in \PGL_{N+1}$
            so that $\phi = \Phi^g$.  Then for any $\sigma \in G_K$ we compute
            \begin{equation*}
                \phi^{f_{\sigma}} = \phi^{\sigma} = (\Phi^g)^{\sigma} =
                (\Phi^{\sigma})^{g^{\sigma}} = (\Phi^{F_{\sigma}})^{g^{\sigma}} =
                \phi^{g^{-1}F_{\sigma}g^{\sigma}}.
            \end{equation*}
            Hence, $g^{-1}F_{\sigma}g^{\sigma}f_{\sigma}^{-1} \in
            \mathcal{A}_{\phi}$.

        \item
            Suppose first that $K$ is a field of definition for $\xi$, so
            there is a map $\Phi \in \xi$ which is defined over $K$. Then
            $\Phi^{\sigma} = \Phi$ for all $\sigma \in G_K$, so in (c) we let
            $F_{\sigma}$ be 1, which gives the desired result.

            Now suppose there is a $g \in \PGL_{N+1}$ such that
            $g^{-1}g^{\sigma}f^{-1}_{\sigma} \in \mathcal{A}_{\phi}$ for all $\sigma \in
            G_K$.  Let $h = g^{-1}g^{\sigma}f_{\sigma}^{-1} \in
            \mathcal{A}_{\phi}$ and let $\Phi=\phi^{g^{-1}} \in \xi$.  Note that $h^{-1} \in \calA_{\phi}$ and $h^{-1}g^{-1} = f_{\sigma}(g^{-1})^{\sigma}$. Then
            \begin{equation*}
                \Phi^{\sigma} = \left(\phi^{g^{-1}}\right)^{\sigma} =
                (\phi^{\sigma})^{(g^{-1})^{\sigma}} =
                \phi^{f_{\sigma}(g^{-1})^{\sigma}} =
                \phi^{h^{-1}g^{-1}}=
               \phi^{g^{-1}} =\Phi.
            \end{equation*}
            Hence, $\Phi$ is defined over $K$, so $K$ is a field of definition for
            $\xi$.\qedhere
        \end{enumerate}
    \end{proof}

    \begin{defn}
        A scheme $X/K$ is called a \emph{Brauer-Severi variety} of dimension $N$
        if there exists a finite, separable field extension $L/K$ such that
        $X \otimes_{K} \Spec L$ is isomorphic to $\mathbb{P}^{N}_L$.  Two
        Brauer-Severi varieties are considered \emph{equivalent} if they are
        isomorphic over $K$.
    \end{defn}

    \begin{prop} \label{prop2}
        \mbox{}
        \begin{enumerate}
            \item \label{prop2_item1} There is a one-to-one correspondence between the set of
              Brauer-Severi varieties of dimension $N$ up to equivalence and the
              cohomology set $H^{1}(G_K,\PGL_{N+1})$.  This correspondence is defined
              as follows:  Let $X/K$ be a Brauer-Severi variety of dimension $N$ and
              choose a $\overline{K}$-isomorphism $j: \mathbb{P}^{N} \to X$.  Then
              the associated cohomology class $c_X \in H^{1}(G_k,\PGL_{N+1})$ is given
              by the cocycle
              \begin{equation*}
                G_K \to \PGL_{N+1}, \quad \sigma \mapsto j^{-1} \circ j^{\sigma}.
              \end{equation*}

            \item  \label{prop2_item2} The following conditions are equivalent:
                \begin{enumerate}
                    \item \label{prop2_item2_1} $X$ is $K$-isomorphic to $\mathbb{P}^{N}$.
                    \item \label{prop2_item2_2} $X(K) \neq \emptyset$.
                    \item \label{prop2_item2_3} $c_X = 1$.
                \end{enumerate}
        \end{enumerate}
    \end{prop}

    \begin{proof}
        Brauer-Severi varieties are discussed by Serre \cite{Serre2} and in more detail by Jahnel
        \cite{Jahnel}.  In particular, (\ref{prop2_item1}) is discussed in \cite[X.6]{Serre2} and the
        equivalence of (\ref{prop2_item2_1}) and (\ref{prop2_item2_3}) follow from (\ref{prop2_item1}).  The equivalence of (\ref{prop2_item2_1}) and (\ref{prop2_item2_2}) is \cite[Exercise X.6.1]{Serre2} and also \cite[Proposition 4.8]{Jahnel}.
    \end{proof}

    \begin{defn}
        The \emph{Brauer group} of a field $K$ is given by
        \begin{equation*}
            \Br(K) \cong H^{2}(G_K,\overline{K}^{\ast}).
        \end{equation*}
    \end{defn}

    \begin{prop} \label{prop3}
        Let $X/K$ be a Brauer-Severi variety of dimension $N$.  If there is a zero-cycle $D$  on $X$
        such that $D$ is defined over $K$ and $\deg(D)$ is relatively
        prime to $N+1$, then $X(K)\neq \emptyset$.
    \end{prop}

    \begin{proof}
        We consider the two exact sequences
        \begin{equation*}
            1 \to \mu_{N+1} \to \text{SL}_{N+1}(\overline{K}) \to
            \PGL_{N+1}(\overline{K}) \to 1
        \end{equation*}
        and
        \begin{equation*}
            1 \to \mu_{N+1} \to \overline{K}^{\ast} \xrightarrow{z \to z^{N+1}}
                \overline{K}^{\ast} \to 1.
        \end{equation*}
        We take their cohomology to determine maps
        \begin{align*}
            0& \to H^{1}(G_K,\PGL_{N+1}) \to H^2(G_K,\mu_{N+1}), \text{ and}\\
             0 &\to H^{2}(G_K,\mu_{N+1}) \to H^{2}(G_K,\overline{K}^{\ast}) \xrightarrow{z \to z^{N+1}}
                H^{2}(G_K,\overline{K}^{\ast}) \to \cdots
        \end{align*}
        Note: we use the facts that $H^{1}(G_k,SL_{N+1}) $ is trivial by \cite[Chapter~X, Corollary to Proposition~3]{Serre2} and
            $H^{1}(G_K,\overline{K}^{\ast}) $ is trivial by Hilbert's Theorem~90.   Hence,
        \begin{equation*}
            H^{1}(G_K,\PGL_{N+1}) \hookrightarrow H^2(G_K,\mu_{N+1}) \cong \Br(K)[N+1]
            \subset \Br(K).
        \end{equation*}

        Let $c_X$ be the cohomology class associated to $X$ as in Proposition~\ref{prop2}.  Choose a Galois
        extension $M/K$ such that $X(M) \neq \emptyset$.  Then by Proposition~\ref{prop2}(b),
        the restriction of $c_X$ to $H^{1}(G_M,\PGL_{N+1})$ is trivial, so
        $c_X$ comes from an element in $H^1(G_{M/K},\PGL_{N+1}(M))$, which we will also denote by~$c_X$.

        For any prime $p$ relatively prime to $N+1$, let $G_p$ denote the $p$-Sylow
        subgroup of $G_{M/K}$, and let $M_p = M^{G_p}$ be the fixed field of $G_p$.
        Taking appropriate inflation and restriction maps gives the following
        commutative diagram:
        \begin{equation*}
            \xymatrix{
                H^{1}(G_K,\PGL_{N+1}) \ar@{^{(}->}[r]^-{1-1} & H^2(G_K,\mu_{N+1})
                \ar[r]^-{\sim}
                & \Br(K)[N+1] \subset \Br(K) \ar[d]_-{\text{Res}}\\
                H^1(G_{M/K},\PGL_{N+1}(M)) \ar[u]_-{\text{Inf}}
                \ar[d]^-{\text{Res}}& & \Br(M_p) \\
                H^1(G_p,\PGL_{N+1}(M)) \ar@{^{(}->}[rr]^-{1-1}& & H^2(G_p,M^{\ast})
                \ar[u]^-{\text{Inf}}_-{1-1}
            }.
        \end{equation*}
        The bottom row comes from taking $G_p $ cohomology of the exact
        sequence
        \begin{equation*}
            1 \to M^{\ast} \to \text{GL}_{N+1} \to \PGL_{N+1}(M) \to 1
        \end{equation*}
        and using the fact that $H^1(G_p,\GL_{N+1}) = 1$ from \cite[X.Proposition 3]{Serre2}.

        If we start with $c_X \in H^1(G_{M/K},\PGL_{N+1}(M))$ and trace it around the
        diagram to $\Br(M_p)$, we find that it has order dividing $N+1$ since it maps
        through $\Br(K)[N+1]$, and it has order a power of $p$ since it maps through
        $H^1(G_p,M^{\ast})$.  Hence, the image of $c_X$ in $\Br(M_p)$ is $0$.  The
        injectivity of the maps along the bottom and up the right-hand side shows
        that $\text{Res}(c_X) = 0$ in $H^1(G_p,\PGL_{M_p})$.  So by Proposition
        \ref{prop2}(b), we have $X(M_p) \neq \emptyset$.

        Let $P \in X(M_p)$ and let
        $P_1,\ldots,P_r$ be the complete set of $M_p/K$ conjugates of $P$. Then $r$ is
        prime to $p$ because it divides $[M_p:K]$, and the zero-cycle $(P_1) + \cdots +(P_r)$
        is defined over $K$.  In other words, there is a zero-cycle $D_p$ on $X$ defined over $K$ whose degree is prime to $p$.

        So for all primes $p$ relatively prime to $(N+1)$, we have a zero-cycle $D_p$ with degree prime to $p$.  Hence, the greatest common divisor of the set
        \begin{equation*}
            \{\deg(D_p) \mid (p,N+1)=1\}
        \end{equation*}
        is a product of divisors of $(N+1)$.

        Given a zero-cycle $D$ defined over $K$ and with $\deg(D)$  relatively prime to
        $N+1$, we can find a finite linear combination
        \begin{equation*}
            E = nD + \sum_{(p,N+1)=1} n_pD_p \qquad \text{with } \deg(E) = 1.
        \end{equation*}
        Since $E$ is defined over $K$, we have found a zero-cycle of degree $1$ that is
        defined over $K$.  Applying classical results on homogeneous spaces (or \cite[Theorem 0.3]{Black}), we have $X(K) \neq \emptyset$.
    \end{proof}

\section{Proof of the Main Theorem} \label{sect.trivial.stabilizer}
    In this section, we focus on dynamical systems $\xi$ with trivial stabilizer; i.e., we assume that $\calA_{\phi} = \{\id \}$ for all $\phi \in \xi$.

    \begin{prop} \label{thm_trivial}
        Let $\xi \in M_d^N(K)$ be a dynamical system with trivial stabilizer.
        \begin{enumerate}
            \item \label{thm_trivial_item1} There is a  cohomology class $c_{\xi} \in H^1(G_K,\PGL_{N+1})$ such that for any $\phi \in \xi$ there is a one-cycle $f: G_K \to \PGL_{N+1}$ in the class of $c_{\xi}$ so that
                \begin{equation*}
                    \phi^{\sigma} = \phi^{f_{\sigma}} \quad \text{for all } \sigma \in G_K.
                \end{equation*}

            \item \label{thm_trivial_item2} Let $X_{\xi}/K$ be the Brauer-Severi variety associated to the cohomology class $c_\xi$.  Then for any $\phi \in \xi$ there exists an isomorphism $i:\P^N \to X_{\xi}$ defined over $\overline{K}$ and a rational map $\Phi: X_{\xi} \to X_{\xi}$ defined over $K$ so that the following diagram commutes:
                \begin{equation}\label{eq4}
                    \xymatrix{\P^N \ar[d]_{\wr}^{i_{/\overline{K}}} \ar[r]^{\phi_{/\overline{K}}}& \P^N \ar[d]_{\wr}^{i_{/\overline{K}}} \\ X_{\xi} \ar[r]^{\Phi_{/K}} & X_{\xi}.}
                \end{equation}

            \item \label{thm_trivial_item3} The following are equivalent.
                \begin{enumerate}
                    \item \label{thm_trivial_item3_1} $K$ is a field of definition for $\xi$.
                    \item \label{thm_trivial_item3_2} $X_{\xi}(K) \neq \emptyset$.
                    \item \label{thm_trivial_item3_3} $c_{\xi}=1$.
                \end{enumerate}
        \end{enumerate}
    \end{prop}
    \begin{proof}
        (\ref{thm_trivial_item1}) Since $\mathcal{A}_{\phi} = \{\id \}$, the result follows immediately from the following:
        \begin{itemize}
        \item
        Proposition \ref{prop1}(\ref{prop1_item1}) says that  $\phi$ determines $f:G_K \to \PGL_{N+1}$.
        \item
        By Proposition \ref{prop1}(\ref{prop1_item2}), $f$ is a one cocycle.
        \item
        By Proposition \ref{prop1}(\ref{prop1_item3}), any other choice of $\phi \in \xi$ gives a cohomologous cocycle.
        \end{itemize}

        (\ref{thm_trivial_item2})  Let $j:\P^N \to X_{\xi}$ be a $\overline{K}$-isomorphism, so $c_{\xi}$ is the cohomology class associated to the cocycle $\sigma \mapsto j^{-1}j^{\sigma}$.  However, from (\ref{thm_trivial_item1}) we know that $c_{\xi}$ is associated to the cocycle $\sigma \mapsto f_{\sigma}$.  Thus, these two cocycles must be cohomologous, meaning there is an element $g \in \PGL_{N+1}$ so that
        \begin{equation*}
            f_{\sigma} = g^{-1}(j^{-1}j^{\sigma})g^{\sigma} \quad \text{ for all } \sigma \in G_K.
        \end{equation*}
        Define $i=jg$, and $\Phi = i\phi i^{-1}$. Then diagram (\ref{eq4}) commutes and it only remains to check that $\Phi$ is defined over $K$.  For any $\sigma \in G_K$, we compute
        \begin{align*}
            \Phi^{\sigma} &= i^{\sigma}\phi^{\sigma}(i^{-1})^{\sigma} =(j^{\sigma}g^{\sigma})(f_{\sigma}^{-1}\phi f_{\sigma})(j^{\sigma}g^{\sigma})^{-1}\\
            &=(j^{\sigma}g^{\sigma}) (g^{-1}j^{-1}j^{\sigma}g^{\sigma})^{-1} \phi (g^{-1}j^{-1}j^{\sigma}g^{\sigma}) (j^{\sigma}g^{\sigma})^{-1}\\
            &= jg \phi g^{-1}j^{-1} = i \phi i^{-1}\\
            &=\Phi.
        \end{align*}

        (\ref{thm_trivial_item3})  The equivalence of (\ref{thm_trivial_item3_2}) and (\ref{thm_trivial_item3_3}) was already proven in Proposition \ref{prop2}(\ref{prop2_item2}).  The equivalence of (\ref{thm_trivial_item3_1}) and (\ref{thm_trivial_item3_3}) follows from Proposition \ref{prop1}(\ref{prop1_item4}) when $\mathcal{A}_{\phi} = \{\id \}$.
    \end{proof}

We now prove  Theorem~\ref{thm.main}.

\begin{rem}
It may seem that to apply Theorem \ref{thm.main} we must test infinitely many cases.  In fact, we need only check $\gcd(D_n, N+1) $ for $n \in [1,\varphi(N+1)]$, where $\varphi$ is the Euler totient function.
\end{rem}

    \begin{proof}[Proof of Theorem \ref{thm.main}]
        Take any $\phi \in \xi$ and choose $i: \P^N \to X_{\xi}$ and $\Phi:X_{\xi} \to X_{\xi}$ as in Proposition \ref{thm_trivial}(\ref{thm_trivial_item2}) so that the diagram in~\eqref{eq4} commutes.

        Let $\Gamma_n \subseteq \P^N \times \P^N$ be the graph of $\Phi^n$ and let $\Delta$ be the diagonal map on $X_\xi$.  Define 
        \[\Per_n(\Phi) = \Delta^*(\Gamma_n);\]
         it denotes the set of periodic points of period $n$ of $\Phi$ taken with multiplicity.

       The zero-cycle $\Per_n(\Phi)$ has degree $D_n$ and is defined over $K$ (see \cite[Proposition 4.17]{HutzDynCyc}).
        Hence, $X_{\xi}$ has a zero-cycle of degree relatively prime to $N+1$.  Thus, from Proposition \ref{prop2}, $X_{\xi}(K) \neq \emptyset$ and by Proposition \ref{thm_trivial}(\ref{thm_trivial_item3}) $K$ is a field of definition for $\xi$.
    \end{proof}

    \begin{rem}
        Silverman \cite{Silverman12} showed that whenever $\xi \in M_d^1(K)$ with $d$ even, then $K$ is a field of definition for $\xi$.  When $\xi$ has trivial stabilizer, this is a special case of Theorem \ref{thm.main}.
    \end{rem}

\section{Examples} \label{sect_examples}
In this section, we demonstrate that both hypotheses in Theorem~\ref{thm.main} are necessary.  First, we give a map where $\gcd(D_n, N+1) > 1$ for all positive integers $n$, and we show that the field of moduli is not a field of definition.  Then, we give a map where $\gcd(D_1, N+1) = 1$ but the stabilizer group is nontrivial, and we show that the field of moduli is not a field of definition.

    The second example is particularly  interesting because it demonstrates a striking difference between  $\Hom_d^1$ and $\Hom_d^N$ for $N>1$. In dimension 1, the \emph{condition} for the field of moduli to be a field of definition is independent of the stabilizer group, but the proof is more complicated in the case of nontrivial stabilizer.  Our example shows that in higher dimension, the condition (not just the proof) must, in fact, be different for maps with nontrivial stabilizer.

\subsection{Example 1}\label{subsec:example1}
    We construct a map $\phi:\P^2 \to \P^2$ with trivial stabilizer group whose field of moduli is $\Q$ but for which $\Q$ is not a field of definition.  This is explained by the fact that  $\gcd(D_n,N+1) = \gcd(D_n, 4) > 1$ for all $n$.
	
	Consider the map
	\begin{equation*}
		\phi = [(x-iz)^4,(y+iz)^4,z^4]:\P^2 \to \P^2.
	\end{equation*}
	
\subsubsection*{Claim 1:} $\phi$ has trivial stabilizer.
	\begin{proof}
		It is clear that any element $f \in \mathcal{A}_{\phi}$ must be of the form
		\begin{equation*}
			\begin{pmatrix}
				a_0 & a_1 & a_2\\
				b_0 & b_1 & b_2\\
				0 & 0 & 1
			\end{pmatrix}.
		\end{equation*}
		Then by direct computation we can see there are no solutions to
		\begin{equation*}
			\phi^f = \phi.
		\end{equation*}
%
	\end{proof}

	\subsubsection*{Claim 2:} $\phi$ has field of moduli $\Q$.
	\begin{proof}
		Notice that $\phi$ is defined over $\Q(i)$. Let $\sigma \in \Gal(\Q(i))$ represent complex conjugation. If we can exhibit an $f$ such that $\phi^f = \phi^{\sigma}$, then we have that the field of moduli is $\Q$. The element
		\begin{equation}\label{eqn:fsigma}
			f=\begin{pmatrix}
				0 & \zeta_6 & 0\\
				\zeta_6 & 0 & 0\\
				0 & 0 & \zeta_6
			\end{pmatrix}
		\end{equation}
		where $\zeta_6$ is a primitive $6\tth$ root of unity, has the necessary property.
		
%
	\end{proof}

	\subsubsection*{Claim 3:} $\Q$ is not a field of definition for $[\phi]$.
	\begin{proof}
        From Proposition \ref{prop1}\ref{prop1_item4} and the fact that  $\calA_\phi = \{\id \}$, we know that $\Q$ is a field of definition for $[\phi]$ if and only if there exists a $g \in \PGL_{N+1}$ such that
        \begin{equation*}
            g^{\sigma}f_{\sigma}^{-1} = g \qquad \forall \sigma \in \Gal(\overline{K}/K).
        \end{equation*}
        In that case, we have $\phi^{g^{-1}}$  defined over $\Q$.

        Since $\Q(i)$ is a field of definition for $\phi$, we only need to consider $\sigma$ to be complex conjugation, so the   associated $f_{\sigma}$ is given by~\eqref{eqn:fsigma}.
 We now show that if $g$ satisfies $g^{\sigma}f_{\sigma}^{-1} = g$, then necessarily $g \in \PGL_3(\Q(i))  \smallsetminus \PGL_3(\Q)$.

            Recall we are  assuming that $g^{-1}$ conjugates $\phi$  to some $\psi\in \Hom_4^2(\Q)$. Let $K$ be the Galois closure of the field of definition of $g$, and let $G = \Gal(K/\Q)$. Write $H= \Gal(K/\Q(i)) \unlhd G$. For $\sigma \in H$, we have
            \begin{align*}
                \psi &= g \circ \phi \circ g^{-1}, \text{ so}\\
              \psi^{\sigma} &= g^{\sigma} \circ \phi^{\sigma} \circ  (g^{-1})^{\sigma}
              = g^{\sigma} \circ \phi \circ (g^{-1})^{\sigma}, \text{ since } \phi \in \Hom_4^2\left(\Q(i)\right).
            \end{align*}
            Since we assume $\psi$ is defined over $\Q$, $\psi^\sigma = \psi$ and we have
             $\{(g^{-1})^{\sigma} \col \sigma \in H\}  $ is contained in the conjugating set
             \[
             \Conj_{\phi,\psi} = \{ f \in \PGL_3 \colon \phi^f = \psi \}.
             \]

 But when $\Conj_{\phi,\psi}$ is non-empty, it is a principle homogeneous space for $\calA_{\phi}$ (see \cite[Section 3]{FMV}). Since $\calA_{\phi}$ is trivial, we conclude that $\#\Conj_{\phi,\psi} \leq 1$. In other words, $(g^{-1})^{\sigma} = g^{-1}$ for all $\sigma \in H$, meaning $g^{-1}\in \PGL_3\left(\Q(i)\right)$.   But clearly $g^{-1} \not \in \PGL_3(\Q)$ since that would imply $\phi = \psi^g$ is defined over $\Q$ as well.

        Now that we know that $g \in \PGL_3(\Q(i))$, we can simply set up a system of equations with rational coefficients from
        \begin{equation*}
            g^{\sigma}f_{\sigma}^{-1} = g \qquad
        \end{equation*}
        by equating each entry in the matrices.
        The resulting system has solutions only of the form
        \begin{equation*}
            \begin{pmatrix}
                0&0&a-ai\\
                0&0&b-bi\\
                0&0&c-ci
            \end{pmatrix}
        \end{equation*}
        which are clearly not elements of $\PGL_3$. Thus, the necessary $g$ does not exist and $\Q$ is not a field of definition.
        \end{proof}

		\subsubsection*{Claim 4:} $\gcd(D_n,N+1) > 1$ for all $n$.
	\begin{proof}
		We have $N=2$ so the only possible divisors of $N+1$ are $1,3$. With $d=4$ we compute
		\begin{equation*}
			D_n = 1 + 4^n + 4^{2n} \equiv 1 + 1 + 1 \equiv 0 \pmod{3}.
		\end{equation*}
	\end{proof}
	We see that $\phi$ is a morphism of $\P^2$ with trivial stabilizer; but since $\gcd(D_n, N+1) > 1$ for all $n$,  the conclusions of Theorem \ref{thm.main} do not apply.

\subsection{Example 2}\label{subsec:example2}
    We construct a map $\phi:\P^2 \to \P^2$ with $\gcd(D_1,N+1) =\gcd(D_1, 3)= 1$, whose field of moduli is $\Q$ but for which $\Q$ is not a field of definition.  This is explained by the fact that   $\mathcal A_\phi$ is nontrivial.
	
	Consider the map
	\begin{equation*}
		\phi = [i(x-y)^3,(x+y)^3,z^3]: \P^2 \to \P^2.
	\end{equation*}
	Note that $D_1 = 1 + 3 + 3^2 = 13$, so indeed $\gcd(D_1,N+1) = 1$.

\subsubsection*{Claim 1:} $\phi$ has field of moduli $\Q$.

\begin{proof}
		Notice that $\phi$ is defined over $\Q(i)$. Let $\sigma \in \Gal(\Q(i))$ represent complex conjugation. If we can exhibit an $f$ such that $\phi^f = \phi^{\sigma}$, then we have that the field of moduli is $\Q$. A calculation shows that
    \begin{equation*}
    f =
    	\begin{pmatrix} 0 &-\zeta_8&0\\\zeta_8&0&0\\0&0&-\zeta_8^2\end{pmatrix}
    \end{equation*}
    has this property. Here $\zeta_8$ is a primitive eighth root of unity.
    \end{proof}

\subsubsection*{Claim 2:}
     $\Q$ is not a field of definition.

\begin{proof}
If $\psi$ is conjugate to this map by $g \in \PGL_3$, then we may assume that $g$ fixes the line $L: \{z=0\}$.
Moreover, the image $g(L)$ must be totally ramified and fixed by $\psi$, and it is the unique line with these properties. Hence, $g(L)$ must be stable under the action of $\Gal(\bar{\Q}/\Q)$ (i.e., defined over $\Q$). So,   conjugating by an element of $ \PGL_3(\Q)$  and renaming $g$ appropriately, we may assume that $g(L) = L$.

Now, restricting $\phi$ to the first two coordinates gives a morphism of $\mathbb{P}^1$ defined by
    \[
    \alpha(z) = i\frac{(z-1)^3}{(z+1)^3}.
 \]
 If $\phi$ has field of definition $\Q$ and we  take the conjugating map $g$ to fix $L$, then restricting $\psi = \phi^g$ to the first two coordinates yields a map conjugate to $\alpha$ and defined over $\Q$.
    But Silverman shows in~\cite{Silverman12}  that $\Q$ is not a field of definition for $\alpha$. So $\Q$ is not a field of definition for $\phi$.
    \end{proof}

	\subsubsection*{Claim 3:} $\phi$ has nontrivial stabilizer.
	
	\begin{proof}
	Let
	\begin{equation*}
		g =
		\begin{pmatrix}
			-1 & 0&0\\0&-1&0\\0&0&1
		\end{pmatrix}.
	\end{equation*}
	It is a simple matter to check that $\phi^g =\phi$, so $g \in  \calA_{\phi)}$.
	\end{proof}

In fact, this  argument shows that
	\begin{equation*}
		\phi = [i(x-y)^d,(x+y)^d,z^d]: \P^2 \to \P^2
	\end{equation*}
	has field of moduli $\Q$ but field of definition at least quadratic over $\Q$ whenever $d$ is odd.

\subsection*{Acknowledgements}
The authors thank Rafe Jones and Joe Silverman for helpful conversations.   Thanks also to Xander Faber for a suggested strategy for the proofs in Section~\ref{subsec:example1} and Section~\ref{subsec:example2}.

\bibliography{auto_bib}
\bibliographystyle{plain}

\end{document}